\documentclass{amsart}[12pt]
\usepackage{amsmath, amsthm, amscd, amsfonts,}
\usepackage{graphicx,float}

\usepackage{amssymb}

\usepackage{pb-diagram}

\DeclareMathOperator{\Col}{Col}

\newtheorem{theorem}{Theorem}[section]
\newtheorem{lemma}[theorem]{Lemma}

\newtheorem{definition}[theorem]{Definition}

\newtheorem{claim}[theorem]{Claim}
\numberwithin{equation}{section}

\usepackage{pb-diagram}

\def\rmark{\mbox{$\rm\bf\rule{0.06em}{1.45ex}\kern-0.05em R$}}
\def\pmark{\mbox{$\rm\bf\rule{0.06em}{1.45ex}\kern-0.05em P$}}
\def\nmark{\mbox{$\rm\bf\rule{0.06em}{1.45ex}\kern-0.05em N$}}
\def\vdash{\mbox{$\rm\| \kern-0.13em -$}}

\begin{document}

\title[On a theorem of Magidor]{On a theorem of Magidor}

\author[Mohammad Golshani]{Mohammad
  Golshani}

\thanks{The  author's research has been supported by a grant from IPM (No. 96030417). The results of this paper were explained by Menachem Magidor to the author during the Sy David Friedman's 60th-Birthday Conference in 2013.} \maketitle



\begin{abstract}
Assuming $\kappa$ is a supercompact cardinal and $\lambda$ is an inaccessible cardinal above it, we present an idea due to Magidor, to find a generic extension in which $\kappa=\aleph_\omega$ and $\lambda=\aleph_{\omega+1}.$

\end{abstract}
\section{introduction}
Suppose $\kappa$ is a supercompact cardinal and $\lambda$ is an inaccessible cardinal above it. There are several ways to find an extension in which $\kappa$ becomes a singular
cardinal  and $\lambda=\kappa^+,$ let us mention at least two:
\begin{itemize}
\item Suppose $\kappa$ is Laver indestructible. Force with $\Col(\kappa, <\lambda)$. In the extension, $\kappa$ remains supercompact and hence measurable, so one can change the cofinality of $\kappa$ to $\omega$ using Prikry's forcing. One can even makes $\kappa=\aleph_\omega,$ preserving cardinals above $\kappa.$ In the final model, $\kappa$ is a singular cardinal of countable cofinality and $\lambda$ is its successor.
\item Merimovich's supercompact extender based Prikry forcing \cite{merimovich} can be  used as well to get an extension in which $cf(\kappa)=\omega$
and $\lambda=\kappa^+.$
\end{itemize}
In this paper we present an idea of Magidor to present another way of making $\kappa=\aleph_\omega$ and $\lambda=\aleph_{\omega+1},$
which is of independent interest.

Thus suppose that $\kappa$ is a supercompact cardinal and let $\lambda$ be the least  inaccessible cardinal above it.  Let $\mathcal{U}$ be a normal measure on $P_{\kappa}(\lambda)=\{P \subseteq \lambda : o.t(P)<\kappa, P\cap \kappa \in \kappa   \}.$

\begin{definition}
(1). Given $P \in P_{\kappa}(\lambda),$ let $\kappa_P=P\cap\kappa$ and $\lambda_P=o.t(P),$

(2). Given $P, Q \in P_{\kappa}(\lambda),$ we define $P \prec Q$ iff $P \subseteq Q$ and $\lambda_P < \kappa_Q,$
\end{definition}
Let $D=\{P \in P_{\kappa}(\lambda): \lambda_P$ is the least inaccessible cardinal above $\kappa_P  \}.$ Then $D\in \mathcal{U}.$ We now define the forcing notion $(\mathbb{P}, \leq, \leq^*)$ as follows:

\begin{definition}
A condition in $\mathbb{P}$ is a finite sequence $\langle P_1, \dots, P_n, f_0, \dots, f_n, A, F \rangle$ where
\begin{enumerate}
\item $P_1 \prec ... \prec P_n$ are in $D,$
\item $f_0 \in Col(\omega_1, <\lambda_{P_1}),$
\item $f_i\in Col(\lambda^+_{P_i}, <\lambda_{P_{i+1}}), i=1, \dots, n-1,$
\item $f_n \in Col(\lambda^+_{P_n}, <\kappa),$
\item $ A \in \mathcal{U}$ and $A \subseteq D,$
\item $\forall P\in A, P_n \prec P$ and $\sup(ran(f_n))<\kappa_P, $
\item $dom(F)=A,$
\item For all $P \in A,~F(P)\in Col(\lambda^+_{P}, <\kappa)$,
\item If $P \prec Q$ are in $A$, then $\sup(ran(F(P))) < \kappa_Q$.
\end{enumerate}

\end{definition}
\begin{definition}
Suppose $\pi=\langle P_1, \dots, P_n, f_0, \dots, f_n, A, F \rangle$ and $\pi'=\langle Q_1, \dots, Q_m, g_0, \dots, g_m, B, G \rangle$
are two conditions in $\mathbb{P}$. Then
\begin{itemize}
\item [(a)] $\pi' \leq \pi$ ($\pi'$ is an extension of $\pi$) iff
\begin{enumerate}
\item $m\geq n,$
\item $Q_i=P_i, i=1, \dots, n,$
\item $Q_i\in A, i=n+1, \dots, m,$
\item $g_i \leq f_i, i=0, \dots, n,$
\item $g_{i} \leq F(Q_{i},), i=n+1, \dots, m,$
\item $B \subseteq A,$

\item For all $P\in B, ~G(P) \leq F(P).$
\end{enumerate}
\item [(b)] $\pi' \leq^* \pi$ ($\pi'$ is a direct or a Prikry extension of $\pi$) iff $\pi' \leq \pi$ and $m=n.$
\end{itemize}
\end{definition}
Let $G$ be $\mathbb{P}-$generic over $V$. Then in $V[G]$ we obtain the following sequences in the natural way:
\begin{enumerate}
\item A $\prec-$increasing sequence $\langle P_n: 0 <n<\omega \rangle$ of elements of $P_{\kappa}(\lambda)$ with $\lambda= \bigcup_{n<\omega}P_n,$
\item A sequence $\langle F_n: n<\omega \rangle$ such that $F_0$ is $Col(\omega_1, <\lambda_{P_1})-$generic over $V$ and for $n>0, F_n$ is $Col(\lambda^+_{P_n}, <\lambda_{P_{n+1}})-$generic over $V$.
\end{enumerate}
The next lemma can be proved as in \cite{magidor}.
\begin{lemma}
\begin{enumerate}
\item [(a)] $(\mathbb{P}, \leq, \leq^*)$ satisfies the Prikry property.

\item [(b)] If $X \in V[G]$ is a bounded subset of $\kappa,$ then $X \in V[\langle F_i: i<n  \rangle]$ for some $n<\omega.$

\item [(c)] $CARD^{V[G]} \cap \kappa=\{\omega, \omega_1, \lambda_{P_n}, \lambda^+_{P_n}: 0<n<\omega   \},$

\item [(d)] $V[G]\models \kappa=\aleph_{\omega}$.
\end{enumerate}
\end{lemma}
But note that in $V[G], \lambda$ is collapsed to $\kappa.$ We now define an inner model of $V[G]$ in which $\lambda$ is preserved. Let
$T=\{(A,F): A\in \mathcal{U}, F$ is a function with domain $A$ and $\forall P \in A,~F(P)\in Col(\lambda^+_P, <\kappa) \}.$

 Define an equivalence relation $\equiv$ on $T$ by $(A,F) \equiv (B,G)$ iff there is a set $C\in \mathcal{U}, C \subseteq A\cap B$ such that for all $P$ in $C$ we have $F(P)=G(P).$ Let $[(A, F)]_{\equiv}$ denote the equivalence class of $(A,F)$ with respect to $\equiv$. Also let
\begin{center}
$G_{\equiv}=\{ [(A, F)]_{\equiv}: \exists \pi\in G, \pi=\langle P_1, \dots, P_n, f_0, \dots, f_n, A, F \rangle   \}.$
\end{center}
Define $V^*$ to be
\begin{center}
$V^*=V[\langle \kappa_{P_n}: 0<n<\omega\rangle, \langle F_n: 0<n<\omega \rangle, G_\equiv].$
\end{center}
Our main theorem is as follows:
\begin{theorem}
\begin{enumerate}
\item [(a)] $V^*\models \kappa=\aleph_{\omega},$

\item [(b)] $\lambda$ remains a cardinal in $V^*,$

\item [(c)] Let  $\mu \in (\kappa, \lambda)$. Then $\langle P_n\cap \mu: 0<n<\omega \rangle \in V^*$,

\item [(d)] All cardinals $\mu \in (\kappa, \lambda)$ are collapsed in $V^*$ into $\kappa$
\end{enumerate}
\end{theorem}
\begin{proof}
(a) is trivial, and (b) can be proved as in \cite{magidor} using permutation arguments. Also (d) follows from (c), so let's prove (c).

Suppose that the condition $\pi\in G$ forces ``for some $\mu,$ with $ \kappa<\mu<\lambda,$ we have $\langle P_n\cap \mu: 0<n<\omega \rangle \notin V^*$''. For $P\in P_{\kappa}(\lambda)$ denote by $\mu_P$ the order type of $P\cap \mu.$ Let $\pi=\langle P_1, \dots, P_n, f_0, \dots, f_n, A, F \rangle.$

Without loose of generality we can assume that for $P\in A, \mu\in P,$ so $\mu_P <\lambda_P.$
\begin{claim}
By extending $F$ if necessary we can assume that for every $P, P' \in A$ with $ \lambda_P=\lambda_{P'}$ and $P\cap \mu \neq P' \cap \mu$ we have $F(P)$ is incompatible with $F(P')$.
\end{claim}
\begin{proof}
Use the fact that for any $Q$ in $A$ with $P, P' \prec Q$, the forcing notion $ Col(\lambda^+_{P}, <\lambda_Q)$ has the property that  below any condition there are at least $2^{\mu_Q}$ incompatible conditions.
\end{proof}
Note that $G$ is generic over $V^*$ with respect to a forcing notion which is a subset of $\mathbb{P}$ and is definable in $V^*$. Denote this set of conditions by $\mathbb{P}^*$. Then $G \subseteq \mathbb{P}^*.$
\begin{claim}
If $\pi' \leq \pi, \pi' \in \mathbb{P}^*, \pi'=\langle Q_1, \dots, Q_m, f_0, \dots, f_m, B, G \rangle,$ then for all $n<i\leq m, Q_i\cap \mu=P_i\cap \mu.$
\end{claim}
\begin{proof}
Using the permutation arguments as in \cite{magidor}, one can show that for any two conditions $\eta, \eta' \in \mathbb{P}^*$ there is some permutation $\sigma$ of $\lambda$ such that $\sigma(\eta)$ is compatible with $\eta'$, and if $\eta, \eta'$ are both extensions of $\pi,$ we can pick $\sigma$ such that $\sigma(\pi)=\pi.$

If the claim fails, then we can find two extensions $\eta, \eta'$ of $\pi$ such that if
$$\eta=\langle Q_1, \dots, Q_m, f_0, \dots, f_m, B, G\rangle$$
 and
 $$ \eta'=\langle Q'_1, \dots, Q'_m, f_0, \dots, f_m, B, G \rangle,$$ then there is $n<i\leq m$ such that $Q_i\cap \mu \neq Q'_i\cap \mu$ and such that for some permutation $\sigma$ we have $\sigma(\pi)=\pi$ and $\sigma(\eta)=\eta'.$ We can extend $\eta$ and $\eta'$ by picking the right member to $B$ to put on the top of both of them, so that we can assume that $Q_m=Q'_m.$ Let $n<j<m$ be maximal such that  $Q_j\cap \mu \neq Q'_j\cap \mu.$ But then $Q=Q_{j+1}=Q'_{j+1}$ satisfies $Q_j \prec Q, Q'_j \prec Q.$ $F_j$ must be an extension of both $F(Q_j)$ and $F(Q'_j)$ ($F$ is the function in the condition $\pi$). But by our assumption about $\pi, F(Q_j)$ and $F(Q'_j)$ are supposed to be incompatible elements of $Col(\lambda^+_{Q_j}, <\lambda_Q).$ Contradiction!.
\end{proof}
Given the above Claim, we can easily see that $\langle P_k\cap \mu: n<k<\omega \rangle$ is in $V^*$. But then trivially the sequence $\langle P_n\cap \mu: 0<n<\omega \rangle$, which is the same sequence with the addition of a finite initial segment, is in $V^*$ as well.
\end{proof}
One can use the model $V^*$ to  prove the following.
\begin{theorem}
The cardinal $\aleph_{\omega+1}^{V^*}=\lambda$ is an inaccessible cardinal in $\text{HOD}^{V^*}$.
\end{theorem}

Mohammad Golshani,
School of Mathematics, Institute for Research in Fundamental Sciences (IPM), P.O. Box:
19395-5746, Tehran-Iran.

E-mail address: golshani.m@gmail.com

URL: http://math.ipm.ac.ir/golshani/
\end{document}